\documentclass[12pt,leqno]{article}
\usepackage{amssymb, amscd, amsmath, amsthm}
\allowdisplaybreaks

\def\beq{\begin{equation}}
\def\eeq{\end{equation}}

\theoremstyle{definition}
\newtheorem{definition}{Definition}

\theoremstyle{plain}
\newtheorem{theorem}{Theorem}

\newtheorem{claim}{Claim}

\newtheorem{conjecture}{Conjecture}
\newtheorem{corollary}{Corollary}
\newtheorem{proposition}{Proposition}

\numberwithin{proposition}{section}
\numberwithin{theorem}{section} \numberwithin{definition}{section}
\numberwithin{claim}{section} \numberwithin{lemma}{section}
\numberwithin{conjecture}{section}
\numberwithin{corollary}{section} \numberwithin{equation}{section}
\numberwithin{example}{section} \numberwithin{remark}{section}

\begin{document}

\title{Sharp results concerning disjoint cross-intersecting families}

\author{Peter Frankl\footnote{R\'enyi Institute, Budapest, Hungary}, Andrey
Kupavskii\footnote{University of Oxford and Moscow Institute of Physics and Technology; Email: {\tt kupavskii@yandex.ru} \ \ The research was supported by the Advanced Postdoc.Mobility grant no. P300P2\_177839 of the Swiss National Science Foundation.}}

\date{}
\maketitle

\begin{abstract}
For an $n$-element set $X$ let $\binom{X}{k}$ be the collection of all its $k$-subsets. Two families of sets $\mathcal A$ and $\mathcal B$ are called cross-intersecting
if $A\cap B \neq \emptyset$ holds for all $A\in\mathcal A$, $B\in\mathcal B$.
Let $f(n,k)$ denote the maximum of $\min\{|\mathcal A|, |\mathcal B|\}$ where the maximum is taken over
all pairs of {\em disjoint}, cross-intersecting families $\mathcal A, \mathcal B\subset\binom{[n]}{k}$.
Let $c=\log_2e$. We prove that
$f(n,k)=\left\lfloor\frac12\binom{n-1}{k-1}\right\rfloor$ essentially iff $n>ck^2$ (cf. Theorem~\ref{th:1.4} for the exact statement).
Let $f^*(n,k)$ denote the same
maximum under the additional restriction that the intersection of all members of both $\mathcal A$ and $\mathcal B$ are empty. For $k\ge5$ and $n\ge k^3$ we show
that $f^*(n,k)=\left\lfloor\frac12\left(\binom{n-1}{k-1}-\binom{n-2k}{k-1}\right)\right\rfloor+1$ and the restriction on $n$ is essentially sharp (cf.
Theorem~\ref{th:5.4}).
\end{abstract}

\section{Introduction}
\label{sec:1}

Let $[n]=\{1,2,3,\dots,n\}$ be the standard $n$-element set and let $2^{[n]}$ denote its power set. A subset $\mathcal F\subset 2^{[n]}$ is called a {\it family}. For $0\le
k\le n$ let $\binom{[n]}{k}=\{F\subset [n]:\, |F|=k\}$. Subsets of $\binom{[n]}{k}$ are called {\it $k$-uniform}. A family $\mathcal F$ is called {\em intersecting} if
$F\cap F'\neq\emptyset$ for all $F,F'\in \mathcal F$. Let us state one of the central results in extremal set theory

\vskip6pt\noindent {\bf Erd\H os--Ko--Rado Theorem ([EKR]).} {\em Suppose that $\mathcal F\subset2^{[n]}$ is intersecting. Then
{\rm (i)} and {\rm (ii)} hold.
\begin{itemize}
\item[{\rm (i)}] $|\mathcal F|\leq 2^{n-1}$.
\item[{\rm (ii)}] Assuming that $\mathcal F$ is $k$-uniform and $n\ge2k$ one has
\end{itemize}}
\begin{equation}
\label{eq:1.1} |\mathcal F|\le\dbinom{n-1}{k-1}.
\end{equation}

We should mention that (i) is a trivial consequence of the fact that $F\in \mathcal F$ implies $([n]\setminus F)\notin\mathcal F$. There are many different proofs for \eqref{eq:1.1}.
E.g., \cite{D}, \cite{Ka2}, \cite{Py}, \cite{FF1}, \cite{HK}, to mention a few.

\begin{definition}
\label{def:1.1} If for some $ x \in[n]$, $x\in F$ for all $F\in\mathcal F$ then $\mathcal F$ is called a {\em star}.

The {\em full star} $\mathcal S_x=\{S\in\binom{[n]}{k}:\, x\in S\}$ shows that \eqref{eq:1.1} is sharp.
\end{definition}

\vskip6pt\noindent {\bf Hilton--Milner Theorem (\cite{HM}).} {\em Suppose that $n > 2k$, $\mathcal F\subset\binom{[n]}{k}$ is intersecting and $\mathcal F$ is not a star.
Then
\begin{equation}\label{eq:1.2}
|\mathcal F|\le\dbinom{n-1}{k-1}-\dbinom{n-k-1}{k-1}+1.
\end{equation}}
\vskip6pt

There are many known proofs for this important result as well. E.g. \cite{FF2}, \cite{F87}, \cite{KZ}, to mention a few.

Trying to prove results about intersecting families one arrives naturally at the following notion.

\begin{definition}
\label{def:1.2}
 Two families $\mathcal A$, $\mathcal B$ are called {\em cross-intersecting} if $A\cap B\neq\emptyset$ for all $A\in\mathcal A$, $B\in\mathcal B$.
\end{definition}

Noting that for $\mathcal A=\mathcal B$ the cross-intersecting property reduces to $\mathcal A$ being intersecting, one can generalize (i) to:
\addtocounter{equation}{-1}
\begin{equation}\label{eq:1.2a}
    |\mathcal A|+|\mathcal B|\le 2^n
\end{equation}
whenever $\mathcal A,\mathcal B\subset2^{[n]}$ are cross-intersecting. Although the bound $|\mathcal A|+|\mathcal B|\le \binom{n}{k}$ is true, for $n > 2k$ it is not
sufficient to derive \eqref{eq:1.1}. However, considering products does the job.

\vskip6pt\noindent
 {\bf Pyber Theorem [Py].} {\em Suppose that $\mathcal A,\mathcal B\subset\binom{[n]}{k}$ are cross-intersecting, $n \geq 2k$.
 Then
 \begin{equation}\label{eq:1.3}
    |\mathcal A|\cdot|\mathcal B|\le \binom{n-1}{k-1}^2.
\end{equation}}

Let us mention \cite{FK1} where a short proof of \eqref{eq:1.1} is given.

The following natural question was first considered in \cite{DF}.

Determine or estimate $f(n)\stackrel{\mathrm{def}}{=}\min\{|\mathcal A|,|\mathcal B|\}$ where $\mathcal A, \mathcal B\subset2^{[n]}$ are cross-intersecting and
$\mathcal A\cap\mathcal B=\emptyset$. The following, rather surprising result was proved in \cite{DF}.
\begin{equation}\label{eq:1.4}
    f(n)<\frac{3-\sqrt5}2\cdot2^n
\end{equation}
and the constant $\frac{3-\sqrt5}2$ is optimal.

The second author was the first to consider the corresponding function $f(n,k)$ for $k$-uniform families.

\setcounter{conjecture}{2}
\begin{conjecture}[\cite{K17}]
\label{conj:1.3} \ \ Suppose that $n>2k\ge4$. Let $f(n,k)=$
$\max\min\{|\mathcal A|,|\mathcal B|\}$ where the maximum is over all disjoint and cross-inter\-sect\-ing families
$\mathcal A, \mathcal B\subset\binom{[n]}{k}$. Then
\begin{equation}\label{eq:1.5}
    f(n,k)=\left\lfloor\frac12\dbinom{n-1}{k-1}\right\rfloor.
\end{equation}
\end{conjecture}

A few months later Huang and independently the present authors disproved \eqref{eq:1.5} for $n<k^2$ (cf. \cite{H}). On the positive side Huang \cite{H} proved
\eqref{eq:1.5} for $n>2k^2$.

In the present paper we determine the range where \eqref{eq:1.5} holds almost completely. Set $c=\log_2e$ and note that $c<3/2$.

\setcounter{theorem}{3}
\begin{theorem}
\label{th:1.4}
 \begin{itemize}
 \item[{\rm (i)}] If $n\ge ck^2+(2-c)k$ then \eqref{eq:1.5} is true.
 \item[{\rm (ii)}] If $n\le ck^2-2ck+1$ then \eqref{eq:1.5} fails.
 \end{itemize}
 \end{theorem}

\section{Tools of proofs}
\label{sec:2}

For a family $\mathcal F$  and a positive integer $l$ define the {\em $l$-shadow} $\sigma^{(l)}(\mathcal F)=\{G:\, |G|=l, \ \exists F\in \mathcal F,\, G\subset F\}$.
In other words, $\sigma^{(l)}(\mathcal F) = \bigcup\limits_{F\in\mathcal F}\binom{F}{l}$.

One of the most important results in extremal set theory is the {\em Kruskal--Katona Theorem}. For given positive integers $m, k, l$, $k>l$ it determines the minimum
of $|\sigma^{(l)}(\mathcal F)|$ where $\mathcal F$ is $k$-uniform and $|\mathcal F|=m$.

It was Daykin's proof \cite{D} of the Erd\H os--Ko--Rado Theorem that established the connection between these two important theorems. To state this connection in the
more general setting of pairs of cross-intersecting families let us define the {\em lexicographic order}, $<_L$ on $\binom{[n]}{t}$, $1\le t\le n$. For $G,H\in
\binom{[n]}{t}$,
\[
 G<_L H\text{ iff }\min\{i:\, i\in G\setminus H\}<\min\{j:\, j\in H\setminus G\}.
 \]
That is, $\{1,\,2,\,77\}<_L\{1,\,3,\,4\}$.

For fixed $t$ and $m$, $1\le m < \binom{n}{t}$ let $\mathcal L(n,t,m)$ denote the {\em initial segment}, the first $m$ subsets of size $t$ in the lexicographic order.

Let now $a,b$ be positive integers, $a+b<n$. Hilton \cite{Hi} observed that $\mathcal A\subset\binom{[n]}{a}$ and $\mathcal B\subset\binom{[n]}{b}$ are
cross-intersecting iff
\[
\mathcal A\cap\sigma^{(a)}(\mathcal B^c)=\emptyset,
\]
where $\mathcal B^c=\{[n]\setminus B:\, B\in\mathcal B\}$ is the family of complements. This permits the following equivalent formulation of  the Kruskal--Katona Theorem.

\vskip6pt\noindent
 {\bf Kruskal--Katona Theorem (\cite{Ka1}, \cite{Kr}).}
{\em Let $a,b$ be positive integers, $n>a+b$. Suppose that $\mathcal A\subset\binom{[n]}{a}$ and $\mathcal B\subset\binom{[n]}{b}$ are cross-intersecting. Then
$\mathcal L(n,a,|\mathcal A|)$ and $\mathcal L(n,b,|\mathcal B|)$ are cross-intersecting too.} \vskip6pt

To get the reader familiar with this formulation let us show an easy consequence that we need later.

\begin{corollary}
\label{cor:2.1}
 Let $n>2k>0$ and suppose that $\mathcal G\subset\binom{[n-1]}{k-1}$ and $\mathcal H\subset\binom{[n-1]}{k}$ are cross-intersecting and $|\mathcal
 G|>\binom{n-1}{k-1} - \binom{n-k}{k-1}$. Then
 \begin{equation}\label{eq:2.1}
    |\mathcal H|\le k-1.
\end{equation}
\end{corollary}

\begin{proof}
 The first $\binom{n-1}{k-1}-\binom{n-k}{k-1}=\binom{n-2}{k-2}+\dots+\binom{n-k}{k-2}$ subsets $G\in\binom{[n-1]}{k-1}$ are $\mathcal L(n-1,\,
 k-1,\,\binom{n-1}{k-1}-\binom{n-k}{k-1})=\{G\in\binom{[n-1]}{k-1}:\, G\cap[k-1]\neq\emptyset\}$. The next $(k-1)$-set is $[k,\, 2k-2]$. The only $k$-subsets
 intersecting {\em all} of them are $[k-1]\cup\{j\}, \, k\le j\le2k-2$.\end{proof}

The proof of Theorem \ref{th:1.4} (i) is based on the following result extending the Hilton--Milner Theorem to two families.

\vskip6pt\noindent
 {\bf M\"{o}rs Theorem (\cite{M}).} {\em Suppose that $\mathcal A,\mathcal B\subset\binom{[n]}{k}$ are cross-intersecting, $n>2k>0$ and neither $\mathcal A$ nor
 $\mathcal B$ is a star. Then
 \begin{equation}\label{eq:2.2}
    \min\{|\mathcal A|, |\mathcal B|\}\le\dbinom{n-1}{k-1}-\dbinom{n-k-1}{k-1}+1.
\end{equation}}

Let us note that for $k=1$ all the problems considered so far are trivial. For $k=2$, and neither $\mathcal A$ nor
 $\mathcal B$ being a star, either $\mathcal A=\mathcal B=\binom{T}{2}$ for a 3-element set $T$, or one of $\mathcal A$ and $\mathcal B$ consists of two pairwise
 disjoint 2-sets and the other is a subset of the four 2-sets intersecting both. Therefore from now on we are going to assume $k\ge 3$.

 In a recent work \cite{FK2} we extended the M\"{o}rs Theorem using the notion of diversity (cf. the definition and statement in
 Section~\ref{sec:4}). This result appears to be essential in establishing the exact value of $f^*(n,k):=\min\bigl\{|\mathcal A|,|\mathcal B|:\,\mathcal A,\mathcal
 B\subset \binom{[n]}{k},\ \mathcal A\cap\mathcal B=\emptyset$, $\mathcal A$  and $\mathcal B$ are cross-intersecting, and neither of them is a star$\bigr\}$.

\setcounter{proposition}{1}
 \begin{proposition}
\label{prop:2.2} For $n>2k$ one has
\begin{equation}\label{eq:2.3}
    f^*(n,k)\ge\left\lfloor\frac12\left(\binom{n-1}{k-1}-\binom{n-2k}{k-1}\right)\right\rfloor+1.
\end{equation}
\end{proposition}

\section{The proof of Theorem \ref{th:1.4} and Proposition \ref{prop:2.2}}
\label{sec:3}

Let $\mathcal A,\mathcal B\subset\binom{[n]}{k}$ be disjoint and cross-intersecting. If one of them (say $\mathcal A$) is a star then there are two cases. For
definiteness suppose $1\in A$ for all $A\in\mathcal A$. Should $1\in B$ hold for all $B\in\mathcal B$ then $|\mathcal A\cup\mathcal B|\le\binom{n-1}{k-1}$ implies
\eqref{eq:1.5}.

If there is some $B_0\in\mathcal B$ with $1\notin B_0$ then the cross-intersecting property implies
\begin{equation}\label{eq:3.1}
    |\mathcal A|\le\dbinom{n-1}{k-1}-\dbinom{n-k-1}{k-1}.
\end{equation}

On the other hand, if neither of $\mathcal A$ and $\mathcal B$ is a star, we may apply the M\"{o}rs Theorem and get the bound \eqref{eq:2.2}. Consequently, if
\[
\dbinom{n-1}{k-1}-\dbinom{n-k-1}{k-1}+1\le\frac12\dbinom{n-1}{k-1}
\]
then \eqref{eq:1.5} holds.
Equivalently
\begin{equation}\label{eq:3.2}
    \dbinom{n-k-1}{k-1}\ge\frac12\dbinom{n-1}{k-1}+1.
\end{equation}
Recall the inequality $1 + x < e^x$ valid for all $x$. Suppose that $n\ge c(k^2-k)+2k-1$. Then
\[
\dbinom{n-1}{k-1}\!\Big/\! \dbinom{n\! -\! k\! -\! 1}{k-1}= \! \prod_{1\le i\le k-1} \! 1+\frac{k}{n-k-i}<\left(\! 1 + \frac{k}{ck(k\! -\! 1)}\! \right)^{\! k-1}\! < e^{\frac1c}=2.
\]
This proves \eqref{eq:3.2} for the corresonding range thereby establishing Theorem~\ref{th:1.4} (i).

Let us now prove that for $n\le c(k-1)^2+1$,
\begin{equation}\label{eq:3.3}
    \dbinom{n-1}{k-1}-\dbinom{n-k}{k-1}>\frac12\dbinom{n-1}{k-1}.
\end{equation}
Equivalently,
\begin{equation}\label{eq:3.4}
    \dbinom{n-k}{k-1}<\frac12\dbinom{n-1}{k-1}.
\end{equation}
We derive \eqref{eq:3.4} from the following chain of equalities and inequalities:
\[
\frac{\dbinom{n-k}{k-1}}{\dbinom{n-1}{k-1}}<\left(\frac{n-k}{n-1}\right)^{k-1}=\left(1-\frac{k-1}{n-1}\right)^{k-1}\le\left(1-\frac{k-1}{c(k-1)^2}\right)^{k-1}
\]
\[
=\left(1-\frac{1}{c(k-1)}\right)^{k-1}<e^{-\frac1c}=\frac12.
\]
To use \eqref{eq:3.3} let us define two families $\mathcal F$ and $\mathcal G$.
\begin{align*}
\mathcal F &=\left\{F\in\binom{[n]}{k}:\ F\cap[k]=\{1\} \text{ or } \{2, 3, \dots, k\}\right\},\\
\mathcal G &=\left\{G\in\binom{[n]}{k}:\ 1\in G, \ G\cap \{2, 3, \dots, k\}\neq\emptyset\right\}.
\end{align*}
Then
\begin{equation}\label{eq:3.5}
    |\mathcal F|+|\mathcal G|=\dbinom{n-1}{k-1}+n-k
\end{equation}
and the families $\mathcal F,\mathcal G$ are cross-intersecting while $\mathcal G$ is intersecting. In view of \eqref{eq:3.3}, $|\mathcal G|>\frac12\binom{n-1}{k-1}$.

If $|\mathcal F|>\frac12\binom{n-1}{k-1}$ then we are done.

Suppose $|\mathcal F|\le\frac12\binom{n-1}{k-1}$ and let $\mathcal G_0\subset\mathcal G$ be an arbitrary subfamily satisfying $|\mathcal
G_0|=\lfloor\frac12\binom{n-1}{k-1}\rfloor+1-|\mathcal F|$. Set $\mathcal A=\mathcal F\cup\mathcal G_0$, $\mathcal B=\mathcal G\setminus\mathcal G_0$. By definition
$|\mathcal A|=\lfloor\frac12\binom{n-1}{k-1}\rfloor+1$ and  \eqref{eq:3.5} implies $|\mathcal B|>\frac12\binom{n-1}{k-1}$ as well. This completes the proof  of
Theorem~\ref{th:1.4}.\hfill $\square$

\medskip
Let us turn to the proof of Proposition~\ref{prop:2.2}.

Let $P,Q\in\binom{[2n]}{k}$ satisfy $P\cap Q=\{2\}$. Define
\[
\mathcal A= \left\{A\in\binom{[n]}{k}:\, 1\in A,\ A\cap Q\neq\emptyset\right\}\cup\{P\},
\]
\[
\mathcal B= \left\{B\in\binom{[n]}{k}:\, 1\in B,\ B\cap P\neq\emptyset\right\}\cup\{Q\}.
\]

It is easy to see that $\mathcal A$ and $\mathcal B$ are cross-intersecting and neither of them is a star. However, they are not disjoint.

Noting $|\mathcal A|=|\mathcal B|$ and $|\mathcal A\cup\mathcal B|=\binom{n-1}{k-1}-\binom{n-2k}{k-1}+2$, one can remove equitably the members of $\mathcal
A\cap\mathcal B$ from exactly one of the two families to obtain $\mathcal A_0\subset\mathcal A$, $\mathcal B_0\subset\mathcal B$,
\[
|\mathcal A_0|=\left\lfloor\frac12\left(\dbinom{n-1}{k-1}-\dbinom{n-2k}{k-1}\right)\right\rfloor+1,
\]
\[
|\mathcal B_0|=\left\lceil\frac12\left(\dbinom{n-1}{k-1}-\dbinom{n-2k}{k-1}\right)\right\rceil+1.
\]

Obviously, $\mathcal A_0$ and $\mathcal B_0$ are cross-intersecting and for $k \geq 2$ neither of them is a star. This concludes the proof of \eqref{eq:2.3}.\hfill $\square$

\section{In the grey zone}
\label{sec:4}

In the previous section we proved Theorem~\ref{th:1.4}. To be more exact, we proved that
\[
f(n,k)=\left\lfloor\frac12\dbinom{n-1}{k-1}\right\rfloor\quad\text{if}\quad\frac12\dbinom{n-1}{k-1}\ge\dbinom{n-1}{k-1}-\dbinom{n-k-1}{k-1}
\]
and also
\[
f(n,k)>\frac12\dbinom{n-1}{k-1}\quad\text{if}\quad\frac12\dbinom{n-1}{k-1}\le\dbinom{n-1}{k-1}-\dbinom{n-k}{k-1}.
\]
Let us try and say something about the ``grey zone'', about the narrow range where
\begin{equation}\label{eq:4.1}
    \dbinom{n-1}{k-1}-\dbinom{n-k}{k-1}<\frac12\dbinom{n-1}{k-1} < \dbinom{n - 1}{k - 1} -\dbinom{n-k-1}{k-1}
\end{equation}
or equivalently
\begin{equation}\label{eq:4.2}
    \dbinom{n-k-1}{k-1}<\frac12\dbinom{n-1}{k-1}<\dbinom{n-k}{k-1}.
\end{equation}

Let us prove that even if \eqref{eq:1.5} fails it is almost true.

\begin{proposition}
\label{prop:4.1}
 Suppose that \eqref{eq:4.1} hold for the pair $(n,k)$, $n>2k$. Then
 \begin{equation}\label{eq:4.3}
    f(n,k)\le\frac12\left(\dbinom{n-1}{k-1}+ n - k - 1\right).
\end{equation}
\end{proposition}

\begin{proof}
Let again $\mathcal A,\mathcal B\subset\binom{[n]}{k}$ be disjoint and cross-intersecting, moreover, $\min\{|\mathcal A|,|\mathcal B|\}\ge\frac12\binom{n-1}{k-1}$. If
both are stars then \eqref{eq:1.5} holds. Suppose now that $\mathcal A$ is a star where $n\in A$ for all $A\in\mathcal A$. In view of \eqref{eq:4.1} we may apply
Corollary~\ref{cor:2.1} with $\mathcal G=\mathcal A(n)=\big\{A\setminus\{n\}:\, A\in\mathcal A\big\}$, $\mathcal H=\mathcal B(\bar n)=\{B\in\mathcal B:\, n\notin B\}$.
This gives $|\mathcal H|\le k-1$. Except for $\mathcal H$, all members of $\mathcal A\cup\mathcal B$ contain $n$. Thus we infer
\[
|\mathcal A\cup\mathcal B|\le\dbinom{n-1}{k-1}+k-1.
\]
Since $\mathcal A$ and $\mathcal B$ are disjoint, \eqref{eq:4.3} follows.

The case that remains is when neither $\mathcal A$ nor $\mathcal B$ is a star. To deal with this case we need the notion of diversity and the extension of M\"{o}rs
Theorem.

\setcounter{definition}{1}
\begin{definition}
\label{def:4.2} For a family $\mathcal F\in 2^{[n]}$ define its {\em diversity} $\gamma(\mathcal F)$ by $\gamma(\mathcal F)=\min\{|\mathcal F(\bar i)|:\, i\in [n]\}$.
\end{definition}

It should be clear that $\mathcal F$ is a star iff $\gamma(\mathcal F)=0$.

\setcounter{theorem}{2}
\begin{theorem}[\cite{FK2}]
\label{th:4.3} Suppose that $\mathcal A,\mathcal B\subset\binom{[n]}{k}$ are cross-intersecting, $n>2k$, $k>3$. Let $u$ be an integer, $3\le u\le k$ and suppose that
\begin{equation}\label{eq:4.4}
    \min\{|\mathcal A|,|\mathcal B|\}\ge\dbinom{n-1}{k-1}-\dbinom{n-u-1}{k-1}+\dbinom{n-u-1}{n-k-1}.
\end{equation}
Then either equality holds in \eqref{eq:4.4} for both $\mathcal A$ and $\mathcal B$ or
\begin{equation}\label{eq:4.5}
    \max\{\gamma(\mathcal A),\gamma(\mathcal B)\}<\dbinom{n-u-1}{n-k-1},
\end{equation}
moreover both families share the same (unique) element of maximum degree.
\end{theorem}

Let us return to the proof of Proposition \ref{prop:4.1}.
We apply Theorem~\ref{th:4.3} with $u=k-1$. If
\[
\min\{|\mathcal A|,|\mathcal B|\}\le\binom{n-1}{k-1}-\binom{n-k}{k-1}+\binom{n-k}{n-k-1}
\]
 then via \eqref{eq:4.1}
 \[
 \min\{|\mathcal A|,|\mathcal B|\}\le\frac12\binom{n-1}{k-1}+n-k-1
\]
follows.

In the opposite case from \eqref{eq:4.4} we derive
\[
\max\{\gamma(\mathcal A),\gamma(\mathcal B)\}\le n-k-1.
\]
By symmetry we may assume
\[
\max\{|\mathcal A(\bar n)|,|\mathcal B(\bar n)|\}\le n-k-1.
\]
By disjointness of $\mathcal A$ and $\mathcal B$ we infer
\[
|\mathcal A(n)| + |\mathcal B(n)|\le\dbinom{n-1}{k-1}.
\]
Thus
\[
\frac{|\mathcal A|+|\mathcal B|}2\le \frac12 \dbinom{n-1}{k-1}+n-k-1
\]
follows.
\end{proof}

\section{Determining $f^*(n,k)$ for $n\ge k(k+5)$, $k\ge5$}
\label{sec:5}

\newcommand{\ma}{\mathcal A}
\newcommand{\mb}{\mathcal B}

Throughout this section let $k\ge5$, $n\ge k(k+5)$ and let $\mathcal A, \mathcal B\subset\binom{[n]}{k}$ be disjoint, cross-intersecting. Moreover we assume that
neither of $\ma$ and $\mb$ is a star.

Since we are trying to determine $f^*(n,k)$, in view of Proposition~\ref{prop:2.2} we may assume that
\begin{equation}\label{eq:5.1}
    \min\{|\mathcal A|,|\mathcal B|\}>\frac12\left(\dbinom{n-1}{k-1}-\dbinom{n-2k}{k-1}\right).
\end{equation}

First we show that the above conditions on $n$ and $k$ guarantee that  Theorem~\ref{th:4.3} can be applied  with some $3\le u\le k$.

\begin{claim}\label{cl:5.1}
\begin{equation}\label{eq:5.2}
\min\{|\mathcal A|,|\mathcal B|\}>\dbinom{n-1}{k-1}-\dbinom{n-4}{k-1}-\dbinom{n-4}{n-k-1}.
\end{equation}
\end{claim}

\begin{proof} One can rewrite the RHS as
\[
\dbinom{n-2}{k-2}+\dbinom{n-3}{k-2}+\dbinom{n-4}{k-2}+\dbinom{n-4}{k-3}=\dbinom{n-2}{k-2}+2\dbinom{n-3}{k-2}<3\dbinom{n-2}{k-2}.
\]
 In view of \eqref{eq:5.1} it is sufficient to show that
 \begin{equation}\label{eq:5.3}
 \dbinom{n-2}{k-2}+\dbinom{n-3}{k-2}+\dots+\dbinom{n-10}{k-2}>6\dbinom{n-2}{k-2}.
\end{equation}

Noting that ${\binom{n-i-1}{k-2}}\Bigm/{\binom{n-i}{k-2}} = 1-\frac{k-2}{n-i}\geq 1-\frac{k-2}{n-8}$ for $2\le i\le 8$,
using $n\ge k^2+5k>(k-2)(k+7)+8$ we infer that the above ratio is always more than
\[
1-\frac{k-2}{(k-2)(k+7)}=1-\frac1{k+7}\ge\frac{11}{12}.
\]
Now \eqref{eq:5.3} follows from $1+\frac{11}{12}+(\frac{11}{12})^2+\dots+(\frac{11}{12})^7=12\cdot\big(1-(\frac{11}{12})^8\big)$ and $(\frac{11}{12})^8<\frac12$.
\end{proof}

Now let $u$ be the maximal integer, $3\le u\le k$ such that
\begin{equation}\label{eq:5.4}
\min\{|\mathcal A|,|\mathcal B|\}>\dbinom{n-1}{k-1}-\dbinom{n-u-1}{k-1}+\dbinom{n-u-1}{n-k-1}.
\end{equation}
In view of Theorem~\ref{th:4.3} we may assume that
\[
\max\{|\ma(\bar 1)|,|\mb(\bar 1)|\}<\dbinom{n-u-1}{n-k-1}.
\]
Since neither $\ma$ nor $\mb$ is a star, $u<k$ follows.

Our next goal is to show that for $A\in\ma(\bar 1)$ and $B\in\mb(\bar 1)$ necessarily
\begin{equation}\label{eq:5.5}
    |A\cap B|=1.
\end{equation}

Suppose for contradiction that $|A\cap B|\ge2$. Set $D=A\cup B$. Note that $|D|=2k-|A\cap B|\le 2k-2$.

The cross-intersecting property implies that every $C\in (\ma(1)\cup\mb(1))$ satisfies $D\cap C\ne\emptyset$.
Consequently,
\begin{footnotesize}\begin{equation}\label{eq:5.6}
|\ma(1)\cup\mb(1)|\le\dbinom{n-1}{k-1}-\dbinom{n-2k+1}{k-1} = \dbinom{n - 1}{k - 1} -\dbinom{n-2k}{k-1}-\dbinom{n-2k}{k-2}.
\end{equation}\end{footnotesize}
By \eqref{eq:5.4} we have
\[
|\ma(\bar 1)|+|\mb(\bar 1)|<\dbinom{n-4}{k-4}.
\]
Comparing with \eqref{eq:5.1} we infer
\begin{equation}
\label{eq:5.7}
2 {n - 4\choose k - 4} > {n - 2k \choose k - 2}.
\end{equation}
However, simple computation shows that this inequality fails for $n \geq k(k + 5)$.

Indeed,
\[
\frac{\dbinom{n-2k}{k-2}}{\dbinom{n-4}{k-4}}=\frac{(n-2k)(n-2k-1)}{(k-2)(k-3)}\prod_{2\le i\le k-3}\left(1-\frac{2(k-1)}{n-2-i}\right).
\]
For $k\ge 5$ and $n \geq k(k+5)$,
$\frac{n-2k-1}{k-3} > \frac{n-2k}{k-2} > k + 3 \geq 8$.
Also $\frac{2(k - 1)}{n - (k - 1)} < \frac2{k + 5}$ and $\left(1 - \frac2{k + 5}\right)^{k - 2} > e^{-2}$.
These show that
\[
\dbinom{n-2k}{k-2}>\dbinom{n-4}{k-4}\left(\frac8e\right)^2,
\]
contradicting \eqref{eq:5.7}. Now \eqref{eq:5.5} is proved.

\begin{claim}\label{cl:5.2}
If $A,A' \in\ma(\bar 1)$ and $B,B'\in\mb(\bar 1)$ then
\[
A\cup B=A'\cup B'.
\]
\end{claim}

\begin{proof}
Set $D=A\cup B$, $D'=A'\cup B'$ and suppose indirectly $D\neq D'$. As we pointed out before, the cross-intersecting property implies both $C\cap D\neq\emptyset$ and
$C\cap D'\neq\emptyset$ for all $C\in\ma(1)\cup\mb(1)$. It is easy to see that the total number of $(k-1)$-subsets $\overset\sim C\subset[2n]$, intersecting both  $D$
and $D'$ is largest if $|D\cap D'|=2k-2$.

In that case the number is
\begin{gather*}
 \dbinom{n-1}{k-1}-\dbinom{n-2k+1}{k-1}+\dbinom{n-2k-1}{k-3}\\=\dbinom{n-1}{k-1}-\dbinom{n-2k}{k-1}-\left(\dbinom{n-2k}{k-2}-\dbinom{n-2k-1}{k-3}\right).
\end{gather*}

It is only a slight difference with respect to \eqref{eq:5.6} and we can get a contradiction in exactly the same way.
\end{proof}

\begin{claim}\label{cl:5.3}
$\min\{|\ma(\bar 1)|, |\mb(\bar 1)|\}=1$.
\end{claim}

\begin{proof}
Suppose the contrary. WLOG let $[2,k+1]\in\ma(\bar 1)$, $[k+1,2k]\in\mb(\bar 1)$. Choose some different $A\in\ma(\bar 1)$ and $B\in\mb(\bar 1)$. By Claim~\ref{cl:5.2},
$A\cup B=[2,2k]$. By \eqref{eq:5.5}, $A\cap B=\{j\}$ for some $j\in [2,2k]$.

In the case $j=k+1$, $A\neq[2,k+1]$ implies $|A\cap[k+1,2k]|\ge2$, contradicting \eqref{eq:5.5}. Suppose by symmetry $2\le j\le k$. Now \eqref{eq:5.5} implies
$[2,k+1]\cap B=\{j\}$ whence $B=\{j\}\cup [k+2,2k]$. From $A\cap B=\{j\}$, $A=[2,k+1]$ follows which is in contradiction with our choice $A\neq[2,k+1]$.
\end{proof}

By symmetry, let us suppose that $|\ma(\bar 1)|=1$. WLOG let $[2,k+1]$ be the unique member of $\ma(\bar 1)$ and let $[k+1,2k]$ be one of the members of $\mb(\bar 1)$.
Using Claim~\ref{cl:5.2} and \eqref{eq:5.5} we infer that
\begin{equation}\label{eq:5.8}
    \mb(\bar 1)\subset\left\{\{j\}\cup[k+2,2k]:\, 2\le j\le  k+1\right\},
\end{equation}
in particular
\begin{equation}\label{eq:5.9}
    |\mb(\bar 1)|\le k.
\end{equation}

Together with
\[
|\ma(1)\cup\mb(1)|\le\dbinom{n-1}{k-1}-\dbinom{n-2k}{k-1}
\]
this implies
\begin{align*}
|\ma|+|\mb| &\le\dbinom{n-1}{k-1}-\dbinom{n-2k}{k-1}+1+|\mb(\bar 1)|\\
 &\le\dbinom{n-1}{k-1}-\dbinom{n-2k}{k-1}+1+k.
\end{align*}

In the case $|\mb(\bar 1)|=1$ we infer
\[
\min\left\{|\ma|,|\mb|\right\}\le\left\lfloor\frac12\left(\dbinom{n-1}{k-1}-\dbinom{n-2k}{k-1}\right)\right\rfloor+1
\]
in accordance with \eqref{eq:2.3}.

Let us show that for $n\ge k^3$ the RHS of the last displayed inequality is the value of $f^*(n,k)$.

\setcounter{theorem}{3}
\begin{theorem}
\label{th:5.4} For $n\ge k^3$
\begin{equation}\label{eq:5.10}
    f^*(n,k)=\left\lfloor\frac12\left(\dbinom{n-1}{n-k}-\dbinom{n-2k}{k-1}\right)\right\rfloor+1.
\end{equation}
\end{theorem}

\begin{proof}
 In view of \eqref{eq:5.10}, all we have to show is  $|\mb(\bar 1)|=1$. Suppose for contradiction  $|\mb(\bar 1)|\ge2$. WLOG $[k+1,2k]$ and $\{k\}\cup[k+2,2k]$ belong
 to  $|\mb(\bar 1)|$. Then for $A\in\ma(1)$ either
 \begin{align*}
& A\cap[k+2,2k]\neq\emptyset\quad\text{or}\\
& A\cap[k+2,2k]=\emptyset\quad\text{but }\{k,k+1\}\subset A.
 \end{align*}
 The number of $(k-1)$-sets $A\subset[2,n]$ satisfying these conditions is
 $
\binom{n-1}{k-1}-\binom{n-k}{k-1}+\binom{n-k-2}{k-3}.
 $
 In view of $|\ma(\bar 1)|=1$, to conclude the proof it is sufficient to show
\begin{small} \begin{equation}\label{eq:5.12}
\dbinom{n-1}{k-1}-\dbinom{n-k}{k-1}+\dbinom{n-k-2}{k-3}<\frac12\left(\dbinom{n-1}{k-1}-\dbinom{n-2k}{k-1}\right)-\frac12.
 \end{equation}\end{small}
 Equivalently,
 \begin{align}\label{eq:5.13}
 %\begin{gathered}
\dbinom{n-2}{k-2}+\dots+&\dbinom{n-k}{k-2}<\\
\notag \dbinom{n-k-1}{k-2}+\dbinom{n-k-2}{k-2}+\dots+&\dbinom{n-2k}{k-2}-\left(2\dbinom{n-k-2}{k-3}+1\right).
%\end{gathered}
%\end{equation}
\end{align}
The last term in brackets is of smaller order of magnitude. E.g., for $n-k>4(k^2-4)$ it is smaller than $\frac{1}{2(k+2)}\binom{n-k-1}{k-2}$.

We break up $\binom{n-k-1}{k-2}$ into $k+1$ equal parts and use $\frac1{k+1}\binom{n-k-1}{k-2}>\frac1{k+1}\binom{n-k-j}{k-2}$ for $j \ge2$. One of these terms we use to compensate for the last term in \eqref{eq:5.13}. Consequently, instead of
\eqref{eq:5.13} it suffices to show that
\[
\dbinom{n-2}{k-2}+\dots+\dbinom{n-k}{k-2}<\frac{k+2}{k+1}\dbinom{n-k-2}{k-2 }+\dots+\frac{k+2}{k+1}\dbinom{n-2k}{k-2}.
\]
This inequality follows once we show
\begin{equation}\label{eq:5.14}
    \dbinom{n-k-j}{k-2}\left/\dbinom{n-j}{k-2}\right.>\frac{k+1}{k+2}\quad\text{for }2\le j\le k.
\end{equation}
Let us expand the LHS of \eqref{eq:5.14} and use the Bernoulli inequality
\[
\frac{\dbinom{n-k-j}{k-2}}{\dbinom{n-j}{k-2}}=\prod_{0\le i\le k-3}\left(1-\frac{k}{n-j-i}\right)>\left(1-\frac{k}{n-2k}\right)^{k-2} >1-\frac{k(k-2)}{n-2k}.
\]
Now $n\ge k^3$ implies $ n-2k\ge k^3-2k> (k+2)k(k-2)$, i.e., the RHS is at least $1-\frac1{k+2}=\frac{k+1}{k+2}$ completing the proof.
\end{proof}

Let us mention that our argument was essentially sharp, that is for $n<(1-\varepsilon)k^3$ the inequality \eqref{eq:5.12} would fail completely. That is the
difference of the two sides would be much more than $k$. Consequently, imitating the proof of Proposition~\ref{prop:2.2} we can show the following.

\setcounter{proposition}{4}
\begin{proposition}\label{prop:5.5}
For any $\varepsilon>0$ and $k>k_0(\varepsilon)$ in the range $k^2+5k<n<(1-\varepsilon)k^3$ one has
\[
f^*(n,k)= \left\lfloor \frac12 \left(\dbinom{n-1}{k-1}-\dbinom{n-2k}{k-1}+k+1\right)\right\rfloor.
\]
\end{proposition}
\hfill $\square$

\end{document}